\documentclass[a4paper,11pt]{article}
\usepackage[latin1]{inputenc}
\usepackage[english]{babel}
\usepackage{amsmath}
\usepackage{amsfonts}
\usepackage{amssymb}
\usepackage{epsfig}
\usepackage{amsopn}
\usepackage{amsthm}
\usepackage{color}
\usepackage{graphicx}
\usepackage{enumerate}
\newtheorem{theorem}{Theorem}[section]
\newtheorem{corollary}[theorem]{Corollary}
\newtheorem{lemma}[theorem]{Lemma}
\newtheorem{proposition}[theorem]{Proposition}

\newtheorem*{theorem*}{Theorem}
\newtheorem*{lemma*}{Lemma}
\newtheorem*{remark*}{Remark}
\newtheorem*{definition*}{Definition}
\newtheorem*{proposition*}{Proposition}
\newtheorem*{corollary*}{Corollary}
\numberwithin{equation}{section}
\newcommand{\real}{\mathbb{R}}

\let\ced=\c         

\def\e{\varepsilon}        

\def\qed{\,\unskip\kern 6pt \penalty 500
\raise -2pt\hbox{\vrule \vbox to8pt{\hrule width 6pt
\vfill\hrule}\vrule}\par}
\definecolor{darkblue}{rgb}{0.05, .05, .65}
\definecolor{darkgreen}{rgb}{0.1, .65, .1}
\definecolor{darkred}{rgb}{0.8,0,0}
\newcommand{\beqn}{\begin{equation}}
\newcommand{\eeqn}{\end{equation}}
\newcommand{\bear}{\begin{eqnarray}}
\newcommand{\eear}{\end{eqnarray}}
\newcommand{\bean}{\begin{eqnarray*}}
\newcommand{\eean}{\end{eqnarray*}}
\begin{document}

\title{\huge \bf Large time behavior for a quasilinear diffusion equation with critical gradient absorption}

\author{
\Large Razvan Gabriel Iagar\,\footnote{Instituto de Ciencias
Matem\'aticas (ICMAT), Nicolas Cabrera 13-15, Campus de Cantoblanco,
Madrid, Spain, \textit{e-mail:}
razvan.iagar@icmat.es},\footnote{Institute of Mathematics of the
Romanian Academy, P.O. Box 1-764, RO-014700, Bucharest, Romania.}
\\[4pt] \Large Philippe Lauren\c cot\,\footnote{Institut de
Math\'ematiques de Toulouse, CNRS UMR~5219, Universit\'e de
Toulouse, F--31062 Toulouse Cedex 9, France. \textit{e-mail:}
Philippe.Laurencot@math.univ-toulouse.fr}\\ [4pt] }
\date{\today}
\maketitle

\begin{abstract}
We study the large time behavior of non-negative solutions to the
nonlinear diffusion equation with critical gradient absorption
$$
\partial_t u-\Delta_{p}u+|\nabla u|^{q_*}=0 \quad \hbox{in} \
(0,\infty)\times\real^N\ ,
$$
for $p\in(2,\infty)$ and $q_*:=p-N/(N+1)$. We show that the
asymptotic profile of compactly supported solutions is given by a
source-type self-similar solution of the $p$-Laplacian equation with suitable logarithmic time and space scales. In the process, we also get optimal decay rates for compactly supported solutions and optimal expansion rates for their supports that strongly improve previous results.
\end{abstract}

\vspace{2.0 cm}


\medskip


\newpage

\section{Introduction and results}\label{sec1}

The goal of the present paper is to study the large time behavior of
non-negative solutions to the following equation which combines a
nonlinear diffusion and a gradient absorption term:
\begin{eqnarray}
\partial_t u-\Delta_p u+|\nabla u|^{q_*} & = & 0 \quad {\rm in} \
(0,\infty)\times\real^N\ , \label{eq1} \\
u(0) & = & u_0 \quad  {\rm in} \ \real^N\ , \label{eq1in}
\end{eqnarray}
where the $p$-Laplacian $\Delta_p u$ is given by $\Delta_p u:={\rm div}(|\nabla u|^{p-2}\nabla u)$ for $p\in (2,\infty)$, and the critical exponent is $q_*:=p-N/(N+1)>1$. We consider the Cauchy problem \eqref{eq1}-\eqref{eq1in} for initial data $u_0$ satisfying
\begin{equation}\label{initdata1}
u_0\in W^{1,\infty}(\real^N), \quad u_0\geq0, \ u_0\not\equiv0,
\quad {\rm supp}\,u_0 \ {\rm is \ compact \ in} \ \real^N.
\end{equation}

The theory for the more general equation
\begin{equation*}
\partial_tu-\Delta_p u+|\nabla u|^q=0, \quad {\rm in} \
(0,\infty)\times\real^N,
\end{equation*}
with general values of the exponents $p$ and $q$, developed very
quickly in the last years, after understanding better how the
competition between the two effects---nonlinear diffusion in form of a $p$-Laplacian term, and first order absorption---works with
respect to different ranges of these exponents.

The semilinear problem, when $p=2$, has been understood first, due
to the possibility of using, at least partially, techniques related
to the heat equation. It has been noticed that there exists two
critical values of the exponent $q$, namely $q=1$ and $q=q_*=(N+2)/(N+1)$, and the qualitative theory together with the large time behavior are now well understood after a series of works \cite{ATU04, BSW02, BL99, BVD13, BGK04, Gi05, GGK03}. More precisely, for $q>q_*$ we have an asymptotic simplification with dominating diffusion, meaning that the absorption plays no role in the large time behavior, while for $1<q<q_*$, there exists a special, unique self-similar solution of \emph{very singular} type, which is the asymptotic profile for the evolution \cite{BKaL04}. The critical case corresponding to $q=q_*$ requires a different treatment and has been investigated in \cite{GL07} using central manifold theory techniques. Finally, the case $q=1$ is the subject of a
well-known open problem (see \cite{BRV97} for the best estimates
available), while for $q\in (0,1)$, finite time extinction occurs \cite{BLS01, BLSS02, Gi05}.

As heat equation techniques are not anymore available for
$p>2$, its study is more involved and it came later. In this range, the qualitative behavior is very different: on the one hand, the nonlinear slow diffusion implies a finite speed of propagation which entails that solutions emanating from compactly supported initial data stay compactly supported for all times. On the other hand, there exists a range of exponents $q\in(1,p-1]$ where the dynamics is governed by the Hamilton-Jacobi part, a new fact that does not appear for $p=2$ and $q>1$; indeed, for $q\in(1,p-1]$, the large time behavior is given by profiles in form of "sandpiles" or "regularized sandpiles", reminding of the solutions to Hamilton-Jacobi equations, see \cite{ILV, LV07}. For higher values of the exponent $q$, asymptotic simplification with dominating diffusion is expected for $q>q_*=p-N/(N+1)$ and the existence of very singular self-similar solutions is known for $q\in(p-1,q_*)$ \cite{Shi}, the lack of a uniqueness result preventing a complete understanding of the large time asymptotics in that case. In the present work we complete the panorama of the slow diffusion case by studying the case $q=q_*$, where finer estimates for the solution than the general ones in \cite{BtL08} are needed.

We finally mention that, for $p\in (1,2)$, a qualitative theory is developed starting from the paper \cite{IL12},
where a new critical exponent $q=p/2$ is discovered, and the
qualitative theory for any $q>0$ is established. In particular,
there are again critical exponents $q=q_*$ and $q=p/2$, limiting
ranges of parameters with different behaviors: diffusive for
$q>q_*$, algebraic decay for $q\in(q_*,p/2)$, exponential decay for
$q=p/2$, finite time extinction for $q<p/2$. Recent work by the
authors helped to establish the existence of special solutions and
the large time behavior for some of these cases, such as
$q\in(q_*,p/2)$, with very singular self-similar solutions, see
\cite{IL13a, IL14}, and $q=p/2$ where \emph{eternal solutions} have
been discovered \cite{IL13b}. The critical case $q=q_*$ is still
open in this range.

\medskip

\noindent \textbf{Main results.} Coming back to Eq.~\eqref{eq1}, our goal is to determine a profile that the solutions approach as
$t\to\infty$. As we are in a critical case which plays the role of
an interface between purely diffusive behavior and mixed-type
behavior, some logarithmic time scales are expected to appear.

We introduce the following constant which will be important in the
analysis
\begin{equation}\label{eta}
\eta:=\frac{1}{N(p-2)+p}=\frac{1}{p(N+1)-2N}.
\end{equation}
We are now in a position to state our main asymptotic result.

\begin{theorem}\label{th.1}
Consider an initial condition $u_0$ satisfying \eqref{initdata1},
$q=q_*=p-N/(N+1)$, and let $u$ be the solution to the Cauchy problem \eqref{eq1}-\eqref{eq1in} with initial condition $u_0$. Then
\begin{equation*}
\lim\limits_{t\to\infty}t^{N\eta}(\log
t)^{p(N+1)\eta}\left|u(t,x)-\frac{1}{t^{N\eta}(\log
t)^{p(N+1)\eta}}B_{A_*}\left(\frac{x}{t^{\eta}(\log
t)^{(2-p)(N+1)\eta}}\right)\right|=0,
\end{equation*}
where
$$
B_A(y)=\left(A-B_0|y|^{p/(p-1)}\right)_+^{(p-1)/(p-2)}, \quad
B_0=\frac{p-2}{p}\eta^{1/(p-1)},
$$
and $A_*$ is uniquely determined and given by:
\begin{equation}\label{interm0}
A_*:=\left(\frac{\displaystyle{(N+1)\int_0^{\infty}\left(1-B_0r^{p/(p-1)}\right)_+^{(p-1)/(p-2)}r^{N-1}\,dr}}
{\displaystyle{\eta^{q/(p-1)}\int_0^{\infty}\left(1-B_0r^{p/(p-1)}\right)_+^{q/(p-2)}r^{N-1+q/(p-1)}\,dr}}\right)^{p(p-2)(N+1)\eta/(p-1)}.
\end{equation}
\end{theorem}

Let us remark that $B_A$ is the well-known Barenblatt profile, the
fundamental solution (and also the asymptotic profile) of the pure
diffusion equation (without the gradient absorption term). The
effect of absorption is seen in the fact that both the time-decay
rate and the expansion of supports are changed with respect to the
$p$-Laplacian equation.

In order to prove Theorem~\ref{th.1}, we have noticed that the
precise time decay rate of the solutions was missing from the
theory. Indeed, it follows from \cite{BtL08} that the $L^1$-norm of
the solutions to the Cauchy problem
\eqref{eq1}-\eqref{eq1in} with initial condition $u_0$ satisfying
\eqref{initdata1} decays with a logarithmic rate as
$t\to\infty$; however, as we will prove, the rate in \cite{BtL08} is
not optimal and it can be improved. Based on that, we find the exact
logarithmic correction for the profile. Once optimal estimates are
obtained, the final step will be an application of the stability
technique of Galaktionov and V\'azquez \cite{GVaz, GVazBook}
together with some analysis concerning the variation of the
$L^1$-norm of the solution, which is needed to establish the
uniqueness (and precise parameter $A=A_*$) of the asymptotic
profile.

\medskip

A by-product of our analysis is the following expansion property of the positivity set of non-negative compactly supported solutions to \eqref{eq1}.

\begin{proposition}\label{pr.posexp}
Consider an initial condition $u_0$ satisfying \eqref{initdata1},
$q=q_*=p-N/(N+1)$, and let $u$ be the solution to the Cauchy problem for the equation \eqref{eq1} with initial condition $u_0$. Introducing the positivity set
$$
\mathcal{P}(t) := \{ x\in\real^N\ :\ u(t,x)>0 \}
$$
of $u$ at time $t\ge 0$, there are $\varrho_2>\varrho_1>0$ such that
\begin{equation*}
B\left( 0 , \varrho_1 t^{\eta}(\log t)^{(2-p)(N+1)\eta} \right)
\subset \mathcal{P}(t) \subset B\left( 0 , \varrho_2 t^{\eta}(\log t)^{(2-p)(N+1)\eta}
\right) \ , \quad t\ge 1\ .
\end{equation*}
\end{proposition}

\medskip

\noindent \emph{Organisation of the paper.} We begin by establishing our optimal time decay rate  for the solutions to \eqref{eq1}, together with the rate of expansion of supports. As we already mentioned, they strongly improve existing results and we consider them as the main theoretical novelty of the present work. This is done in Section~\ref{sec.sd}, making use of a fine analysis of the supports and functional inequalities of H\"older and Poincar\'e type. The precise estimates are given in Proposition~\ref{prop.improv} and Corollary~\ref{cor.est}. These sharp estimates give in particular the correct time scales, thus allowing for a rescaling to reach a new equation with solutions being uniformly bounded in $L^1(\real^N)\cap W^{1,\infty}(\real^N)$. The scaling step is performed in Section~\ref{sec.scaling}, where we also show that the Barenblatt profile itself provides a suitable subsolution, which is fundamental in the sequel in avoiding the phenomenon of collapse to zero in the limit. Finally, in Section~\ref{sec.final}, we complete the proof of Theorem~\ref{th.1} by using the stability theorem for small perturbations of dynamical systems from \cite{GVaz, GVazBook}, which we briefly recall in a more abstract framework in the Appendix. This approach is by now rather standard; nevertheless, we present it in some details, as the presence of the gradient absorption term leads to some further technical complications.

\section{Optimal $L^{\infty}$-bounds}\label{sec.sd}

Let $u_0$ be an initial condition satisfying \eqref{initdata1} and denote the corresponding solution to the Cauchy problem \eqref{eq1} by $u$. We recall that the existence and uniqueness of a non-negative (viscosity) solution $u\in BC([0,\infty)\times\real^N)$ to \eqref{eq1} which is also a weak solution follows from \cite[Theorem~1.1]{BtL08}, and
moreover it satisfies
\begin{equation*}
0\leq u(t,x)\leq\|u_0\|_{\infty}, \quad \|\nabla
u(t)\|_{\infty}\leq\|\nabla u_0\|_{\infty},
\end{equation*}
for any $t\geq 0$. In addition, it follows from \cite[Theorem~1.1 \& Theorem~1.2]{ATU04} and \cite[Theorem~1.6 \&
Corollary~1.7]{BtL08} that $u(t)$ is compactly supported for any
$t>0$ so that the maximal radius of the (compact) support of $u(t)$ and $u_0$ defined by
$$
\varrho(t):=\inf\{R>0: u(t,x)=0 \ {\rm for} \ |x|>R\}, \quad
\varrho_0:=\varrho(0),
$$
are finite for each $t\geq 0$. Furthermore, the following estimates for the support and the $L^1$-norm are established in the above mentioned references: there is $C_1>0$ such that
\begin{equation}\label{supp.est1}
\varrho(t)\leq C_1(1+t)^{\eta} \quad {\rm for \ any} \ t\geq0
\end{equation}
and
\begin{equation}\label{l1.est1}
\|u(t)\|_1\leq C_1\log(1+t)^{-1/(q_*-1)} \quad {\rm for \ any} \
t\geq1.
\end{equation}
However, as we shall see below, the estimates \eqref{supp.est1} and
\eqref{l1.est1} are far from being optimal and can be strongly
improved. We finally recall that, due to \cite[Lemma~4.1]{ATU04}, the solution $u$ to \eqref{eq1} cannot vanish identically in finite time, that is,
\begin{equation*}
u(t) \not\equiv 0 \;\;\text{ for }\;\; t\ge 0\ . 
\end{equation*}

\begin{proposition}\label{prop.improv}
There exists a positive constant $C_2$, depending only on $p$, $q$,
$N$ and $u_0$, such that for any $t>0$,
\begin{equation}\label{supp.est2}
\varrho(t)\leq C_2(1+t)^{\eta}(\log(1+t))^{-\eta(p-2)(N+1)},
\end{equation}
and
\begin{equation}\label{l1.est2}
\|u(t)\|_1\leq C_2(\log(1+t))^{-(N+1)}.
\end{equation}
\end{proposition}

Notice that this is a real improvement in the second estimate
\eqref{l1.est2} with respect to \eqref{l1.est1}, since
$$
\frac{1}{q_*-1}=\frac{N+1}{p(N+1)-(2N+1)} < \frac{N+1}{2(N+1)-(2N+1)}=N+1,
$$
for $p>2$.
\begin{proof}
Since $u_0$ is non-negative, continuous and compactly supported in
$\real^N$, there exists a non-negative, continuous and radially
symmetric, radially non-increasing function $U_0$ with compact
support, such that $0\leq u_0(x)\leq U_0(x)$ for any $x\in\real^N$.
Let $U$ be the corresponding solution to \eqref{eq1} with initial
data $U_0$; it follows from \cite[Theorem~1.2]{BtL08} that the
function $x\mapsto U(t,x)$ is also radially symmetric, radially
non-increasing and compactly supported, for any $t>0$, and we deduce
from the comparison principle that $0\leq u(t,x)\leq U(t,x)$, for
any $(t,x)\in(0,\infty)\times\real^N$. Let, for $t\geq0$,
$$
\sigma(t):=\sup\limits_{x\in\real^N}\{|x|: U(t,x)>0\}
$$
be the radius of the support of $U$. We then have
$\varrho(t)\leq\sigma(t)$. Furthermore, we infer from \cite[Theorem~1.2 \& Proposition~1.4]{BtL08} that
\begin{equation}\label{interm1}
|\nabla U^{(p-2)/(p-1)}(t,x)|\leq
C_3\|U(s)\|_{\infty}^{(p-2)/p(p-1)}(t-s)^{-1/p}, \quad {\rm for} \
0\leq s<t
\end{equation}
and
\begin{equation}\label{interm2}
\|U(t)\|_{\infty}\leq C_3\|U(s)\|_1^{p\eta}(t-s)^{-N\eta}, \quad
{\rm for} \ 0\leq s<t.
\end{equation}
It next follows from \eqref{eq1} that, for any non-negative function
$y\in C^1([0,\infty))$, we have
\begin{equation*}
\begin{split}
\frac{d}{dt}\int_{\{|x|\geq y(t)\}}U(t,x)\,dx&=\int_{\{|x|\geq
y(t)\}}\partial_tU(t,x)\,dx-y'(t)\int_{\{|x|=y(t)\}}U(t,x)\,dx\\
&\leq\int_{\{|x|\geq y(t)\}}{\rm div}(|\nabla U|^{p-2}\nabla
U)(t,x)\,dx-y'(t)\int_{\{|x|=y(t)\}}U(t,x)\,dx\\
&\leq-\int_{\{|x|=y(t)\}}|\nabla U(t,x)|^{p-2}\nabla
U(t,x)\cdot\frac{x}{|x|}\,dx\\&-y'(t)\int_{\{|x|=y(t)\}}U(t,x)\,dx\\
&\leq\int_{\{|x|=y(t)\}}\left[\frac{p-1}{p-2}|\nabla U^{(p-2)/(p-1)}(t,x)|^{p-1}-y'(t)\right] U(t,x)\,dx.
\end{split}
\end{equation*}
Fix now $t_0\geq0$. For $t>t_0$ we deduce from the estimates
\eqref{interm1} (with $s=(t+t_0)/2$) and \eqref{interm2} (with
$((t+t_0)/2,t_0)$ instead of $(t,s)$) that
\begin{equation*}
\begin{split}
\frac{p-1}{p-2}|\nabla& U^{(p-2)/(p-1)} (t,x)|^{p-1}\leq2^{(p-1)/p}\frac{p-1}{p-2}C_3^{p-1}\left\|U\left(\frac{t+t_0}{2}\right)\right\|_{\infty}^{p-2)/p} (t-t_0)^{-(p-1)/p}\\
&\leq(2C_3^p)^{(p-1)/p} \frac{p-1}{p-2} \left[C_3\|U(t_0)\|_1^{p\eta}\left(\frac{t-t_0}{2}\right)^{-N\eta}\right]^{(p-2)/p}(t-t_0)^{-(p-1)/p}\\
&\leq\eta C_4\|U(t_0)\|_1^{(p-2)\eta}(t-t_0)^{\eta-1}.
\end{split}
\end{equation*}
Choosing
$$
y(t):=\sigma(t_0)+C_4\|U(t_0)\|_1^{(p-2)\eta}(t-t_0)^{\eta},
$$
the above inequality reads
$$
\frac{p-1}{p-2}|\nabla U^{(p-2)/(p-1)} (t,x)|^{p-1}\leq y'(t),
$$
for any $t\geq t_0$, from which we deduce that
$$
\frac{d}{dt}\int_{\{|x|\geq y(t)\}}U(t,x)\,dx\leq0 \quad {\rm for}
\ t\geq t_0.
$$
Since $y(t_0)=\sigma(t_0)$, we end up with
$$
\int_{\{|x|\geq y(t)\}} U(t,x)\,dx\leq\int_{\{|x|\geq\sigma(t_0)\}}U(t_0,x)\,dx=0,
$$
for any $t>t_0$. Owing to the non-negativity of $U$, this is only possible if $U(t,x)\equiv0$ for $|x|\geq y(t)$, which means
\begin{equation}\label{interm3}
\sigma(t)\leq
y(t)=\sigma(t_0)+C_4\|U(t_0)\|_1^{(p-2)\eta}(t-t_0)^{\eta},
\end{equation}
for $t>t_0$. In order to proceed further, we need the following Poincar\'e inequality:

\begin{lemma}\label{lem.ineq}
Given $\mu\in[1,\infty)$ and $R>0$ there exists a constant $K$ depending only in $N$ and $\mu$ such that
\begin{equation*}
\|w\|_{L^1(B(0,R))}^{\mu}\leq KR^{\mu(N+1)-N}\|\nabla
w\|_{L^{\mu}(B(0,R))}^{\mu} \;\;\text{ for all }\;\; w\in
W_0^{1,\mu}(B(0,R))\ .
\end{equation*}
\end{lemma}

\begin{proof}[Proof of Lemma~\ref{lem.ineq}]
This follows by combining H\"older and Poincar\'e inequalities. More
precisely, setting $\mu' := \mu/(\mu-1)$, one has
\begin{equation*}
\begin{split}
\|w\|_{L^1(B(0,R))}^{\mu}&\leq\|w\|_{L^{\mu}(B(0,R))}^{\mu} | B(0,R)|^{\mu/\mu'}\\
&=K \|w\|_{L^{\mu}(B(0,R))}^{\mu} R^{N(\mu-1)}\\
&\leq KR^{\mu}\|\nabla
w\|_{L^{\mu}(B(0,R))}^{\mu}R^{N\mu-N}=KR^{\mu(N+1)-N}\|\nabla
w\|_{L^{\mu}(B(0,R))}^{\mu},
\end{split}
\end{equation*}
where we have used that the Poincar\'e constant is of order $O(R)$.
\end{proof}

In order to continue the proof of Proposition~\ref{prop.improv}, we first note that integrating \eqref{eq1} over $\real^N$ gives
$$
\frac{d}{dt}\|U(t)\|_1+\|\nabla U\|_{q_*}^{q_*}=0\ .
$$
Since $U(t)$ is supported in $B(0,\sigma(t))$ for each $t\ge 0$, it belongs to $W_0^{1,q_*}(B(0,\sigma(t)))$ and we apply Lemma~\ref{lem.ineq} to obtain
\begin{equation}\label{interm5}
\frac{d}{dt}\|U(t)\|_1+\frac{1}{K}\frac{\|U(t)\|_1^{q_*}}{\sigma(t)^{1/\eta}}\le 0
\end{equation}
for any $t\ge 0$. Now, fix $T\geq1$ and introduce the notation
$$
\Sigma(T):=\sup\limits_{t\in[1,T]}\left\{t^{-\eta}(\log
t)^A\sigma(t)\right\}, \quad A:=\eta(p-2)(N+1).
$$
We infer from \eqref{interm5} that, for any
$t\in[1,T]$,
$$
\frac{d}{dt}\|U(t)\|_1+\frac{(\log
t)^{(p-2)(N+1)}}{Kt}\frac{\|U(t)\|_1^{q_*}}{(t^{-\eta}(\log
t)^A\sigma(t))^{1/\eta}}\leq0,
$$
whence
$$
\frac{d}{dt}\|U(t)\|_1+\frac{(\log
t)^{(p-2)(N+1)}}{Kt(\Sigma(T))^{1/\eta}}\|U(t)\|_{1}^{q_*}\leq0\ , \quad t\in[1,T]\ .
$$
Integrating the above inequality over $(1,t)$, $t\in (1,T)$, we find
\begin{equation}\label{interm6}
\|U(t)\|_1\leq (K(N+1))^{1/(q_*-1)} \Sigma(T)^{1/\eta(q_*-1)}(\log t)^{-(N+1)}, \quad t\in (1,T].
\end{equation}

Next, let $m\ge 1$ be  an integer to be determined later and consider $t\in (1,T]$.

\noindent $\bullet$ Either $t\leq2^m$ and it follows from
\eqref{interm3} with $t_0=1$ that
$$
t^{-\eta}(\log
t)^{A}\sigma(t)\leq(\log2^m)^{A}\left[\sigma(1)+C_4\|U(1)\|_1^{(p-2)\eta}2^{m\eta}\right]\leq C(m)\ .
$$

\noindent $\bullet$ Or $2^m\leq t\leq T$ (if $2^m\le T$), and we infer from \eqref{interm3} (with $t_0=t/2$) and \eqref{interm6} that
\begin{equation*}
\begin{split}
t^{-\eta}(\log t)^A\sigma(t)&\leq t^{-\eta}(\log
t)^A\sigma\left(\frac{t}{2}\right)+C_4t^{-\eta}(\log
t)^A\left\|U\left(\frac{t}{2}\right)\right\|_1^{(p-2)\eta}\left(\frac{t}{2}\right)^{\eta}\\
&\leq2^{-\eta}\left(\frac{t}{2}\right)^{-\eta}\left(\frac{\log
t}{\log(t/2)}\right)^A\log\left(\frac{t}{2}\right)^A\sigma\left(\frac{t}{2}\right)\\
&+C_42^{-\eta}(\log
t)^A\left[\Sigma(T)^{1/\eta(q_*-1)}\left(\log\left(\frac{t}{2}\right)\right)^{-(N+1)}\right]^{(p-2)\eta}.
\end{split}
\end{equation*}
Since $t\geq2^m$, we obtain that
$$
\log t\leq\frac{m}{m-1}\log\left(\frac{t}{2}\right),
$$
whence, plugging this estimate in the previous inequality, we get
$$
t^{-\eta}(\log
t)^A\sigma(t)\leq\left(\frac{m}{m-1}\right)^A2^{-\eta}\left[\Sigma(T)+C_4\Sigma(T)^{(p-2)/(q_*-1)}\right].
$$
Combining the previous two estimates and taking the supremum over $t\in (1,T]$, we obtain
$$
\Sigma(T)\leq\left(\frac{m}{m-1}\right)^A2^{-\eta}\Sigma(T)+C_5\Sigma(T)^{(p-2)/(q_*-1)}+C(m).
$$
We now fix $m$ large enough such that
$$
\delta_m := \left(\frac{m}{m-1}\right)^A2^{-\eta}<1\ .
$$
Taking into account that $p-2<q_*-1$, we deduce from Young's inequality that
$$
\Sigma(T)\leq \delta_m \Sigma(T) + \frac{1-\delta_m}{2}\Sigma(T) +C(m)\ ,
$$
which readily implies that $\Sigma(T)\leq C(m)$ for each $T\geq1$, the constant $C(m)$ being independent of $T$. We have thus proved that
$$
\sigma(t)\leq Ct^{\eta}(\log t)^{-\eta(p-2)(N+1)}, \quad t\geq 1,
$$
while the uniform bound for $\Sigma(T)$ together with
\eqref{interm6} imply
$$
\|U(t)\|_1\leq C(\log t)^{-(N+1)}.
$$
Recalling that $u(t,x)\leq U(t,x)$ for any
$(t,x)\in(0,\infty)\times\real^N$ and that
$\varrho(t)\leq\sigma(t)$, we obtain the expected estimates \eqref{supp.est2} and \eqref{l1.est2}.
\end{proof}

\begin{corollary}\label{cor.est}
There exists a constant $C_5>0$ depending only on $p$, $q$, $N$ and
$u_0$, such that for any $t>0$, we have
\begin{equation}\label{linf.est}
\|u(t)\|_{\infty}\leq C_5(1+t)^{-N\eta}(\log(1+t))^{-p\eta(N+1)}
\end{equation}
and
\begin{equation}\label{linf.grad.est}
\|\nabla u(t)\|_{\infty}\leq
C_5(1+t)^{-(N+1)\eta}(\log(1+t))^{-2\eta(N+1)}.
\end{equation}
\end{corollary}

\begin{proof}
We combine the estimates in \cite[Proposition~1.4]{BtL08} with the
previous estimates of Proposition~\ref{prop.improv}. We thereby obtain
\begin{equation*}
\begin{split}
\|u(t)\|_{\infty}&\leq
C\left\|u\left(\frac{t}{2}\right)\right\|_1^{p\eta}t^{-N\eta}2^{N\eta}\\
&\leq Ct^{-N\eta}\log\left(1+\frac{t}{2}\right)^{-(N+1)p\eta}\leq
Ct^{-N\eta}\left(\frac{1}{2}\log(1+t)\right)^{-(N+1)p\eta}\\
&\leq C(1+t)^{-N\eta}(\log(1+t))^{-(N+1)p\eta}
\end{split}
\end{equation*}
and
\begin{equation*}
\begin{split}
\|\nabla u(t)\|_{\infty}&\leq
C\left\|u\left(\frac{t}{2}\right)\right\|_1^{2\eta}t^{-\eta(N+1)}\\
&\leq Ct^{-\eta(N+1)}(\log(1+t))^{-2\eta(N+1)},
\end{split}
\end{equation*}
which completes the proof.
\end{proof}

As we shall see in the next section, these estimates are optimal and are building blocks in identifying the large time behavior of compactly supported solutions to \eqref{eq1} and thus proving Theorem~\ref{th.1}.

\section{Scaling variables}\label{sec.scaling}

Let $u_0$ be an initial condition satisfying \eqref{initdata1} and
denote the corresponding solution to the Cauchy problem
\eqref{eq1}-\eqref{eq1in} by $u$. According to the
estimates derived in Section~\ref{sec.sd}, we introduce the
following new variables $(s,y)$ and function $w$:
\begin{equation}\label{resc.1}
s=\log(e+t), \quad y=x(e+t)^{-\eta}\log(e+t)^{(p-2)(N+1)\eta},
\end{equation}
and
\begin{equation}\label{resc.2}
u(t,x)=(e+t)^{-N\eta}\log(e+t)^{-p\eta(N+1)} w(s,y).
\end{equation}
By \eqref{eq1} the rescaled function $w$ solves
\begin{subequations}
\begin{eqnarray}
\partial_s w -\mathcal{L}w & = & 0  \quad {\rm in} \
(1,\infty)\times\real^N\ , \label{eq.resc} \\
w(1) & = & u_0 \quad  {\rm in} \ \real^N\ , \label{eq.rescin}
\end{eqnarray}
where
\begin{equation}\label{eq.rescL}
\begin{split}
\mathcal{L}z & := \eta y\cdot\nabla z + \eta Nz + \Delta_p z \\
& \qquad -\frac{1}{s}\left[|\nabla z|^{q_*}-p\eta(N+1)z+(p-2)\eta(N+1)y\cdot\nabla z\right]\ .
\end{split}
\end{equation}
\end{subequations}
For further use, we introduce the autonomous counterpart of \eqref{eq.resc} which is
\begin{subequations}
\begin{equation}
\partial_s v - L v = 0  \quad {\rm in} \ (1,\infty)\times\real^N\ , \label{eq.auton}
\end{equation}
with
\begin{equation}\label{eq.autonL}
Lz:=\eta y\cdot\nabla z+\eta Nz+\Delta_p z.
\end{equation}
\end{subequations}
The boundedness of $w$ readily follows from the previous section.

\begin{lemma}\label{lem.above}
There is a positive constant $C_6$ depending only on $p$, $q$, $N$, and $u_0$ such that
\begin{equation*}
\|w(s)\|_{1}+\|w(s)\|_{\infty}+\|\nabla w(s)\|_{\infty}\leq C_6,
\end{equation*}
for any $s>1$. Moreover, the support of $w(s)$ is localized: there exists $R_0>0$ such that ${\rm supp}(w(s))\subseteq
B(0,R_0)$ for any $s\in(1,\infty)$.
\end{lemma}

\begin{proof}
The bounds for the $W^{1,\infty}$-norm of $w(s)$ are immediate
consequences of estimates \eqref{linf.est} and
\eqref{linf.grad.est}, taking into account the definition of $w$ in
\eqref{resc.2}. The estimate for the $L^1$-norm follows from
\eqref{l1.est2} by a change of variable as below:
\begin{equation*}
\begin{split}
\|w(s)\|_{1}&=(e+t)^{N\eta}\log(e+t)^{p\eta(N+1)}\int_{\real^N} u\left( t,y (e+t)^\eta \log(e+t)^{-(p-2)(N+1)\eta} \right) \,dy\\
&=\log(e+t)^{N+1}\|u(t)\|_1\leq C_2.
\end{split}
\end{equation*}
Finally, the assertion about the localization of the support follows
from estimate \eqref{supp.est2} and the definition of the new
variable $y$ in \eqref{resc.1} (with $R_0=C_2$).
\end{proof}

Lemma~\ref{lem.above} provides a fine upper bound for $w$ but its optimality can only be guaranteed by a lower bound of the same order. In addition, such a lower bound would prevent the possibility of collapsing to the trivial solution in the limit $s\to\infty$. This is a consequence of the following result.

\begin{lemma}\label{lem.below}
There exists $A_{sub}>0$ sufficiently small such that the Barenblatt
profile
\begin{equation*}
B_A(y)=\left(A-B_0|y|^{p/(p-1)}\right)_{+}^{(p-1)/(p-2)}, \quad
B_0=\frac{p-2}{p}\eta^{1/(p-1)},
\end{equation*}
is a subsolution to Eq.~\eqref{eq.resc} for $A\in (0,A_{sub})$.
\end{lemma}

\begin{proof}
A simple computation shows that $LB_A=0$ for any $A>0$, where $L$ is the autonomous operator defined in \eqref{eq.autonL}. Moreover,
$$
\nabla B_A(y)\cdot y=-\frac{B_0p}{p-2}\left(A-B_0|y|^{p/(p-1)}\right)_{+}^{1/(p-2)}|y|^{p/(p-1)},
$$
hence
\begin{equation*}
\begin{split}
- s\mathcal{L}B_A&=|\nabla B_A|^{q_*}+\eta(p-2)(N+1)y\cdot\nabla
B_A -\eta(N+1)p B_A\\
&=\left(\frac{B_0p}{p-2}\right)^{q_*}\left(A-B_0|y|^{p/(p-1)}\right)_{+}^{q_*/(p-2)}|y|^{q_*/(p-1)}\\
&-\eta(N+1)pB_0 \left(A-B_0|y|^{p/(p-1)}\right)_{+}^{1/(p-2)}|y|^{p/(p-1)}\\
&-\eta(N+1)p\left(A-B_0|y|^{p/(p-1)}\right)_{+}^{(p-1)/(p-2)}\\
&=\eta^{q_*/(p-1)}\left(A-B_0|y|^{p/(p-1)}\right)_{+}^{q_*/(p-2)}|y|^{q_*/(p-1)} \\
& -\eta(N+1)pA\left(A-B_0|y|^{p/(p-1)}\right)_{+}^{1/(p-2)}\\
&=\eta\left(A-B_0|y|^{p/(p-1)}\right)_{+}^{1/(p-2)}\\
&\times\left[\eta^{(q_*-p+1)/(p-1)}\left(A-B_0|y|^{p/(p-1)}\right)_{+}^{(q_*-1)/(p-2)}|y|^{q_*/(p-1)}-(N+1)pA\right].
\end{split}
\end{equation*}
Either $|y|\ge (A/B_0)^{(p-1)/p}$ and $-s\mathcal{L}B_A(y)=0$. Or $|y|\leq(A/B_0)^{(p-1)/p}$ and, since $A-B_0|y|^{p/(p-1)}\leq A$, we find
\begin{equation*}
\begin{split}
& -s\mathcal{L}B_A \\
& \quad \leq\eta \left(A-B_0|y|^{p/(p-1)}\right)_{+}^{1/(p-2)} \left[\eta^{(q_*-p+1)/(p-1)}A^{(q_*-1)/(p-2)}\left(\frac{A}{B_0}\right)^{q/p}-(N+1)pA\right]\\
& \quad \leq\eta A \left(A-B_0|y|^{p/(p-1)}\right)_{+}^{1/(p-2)} \left[\eta^{(q_*-p+1)/(p-1)}B_0^{-q_*/p} A^{\theta}-(N+1)p\right]\ ,
\end{split}
\end{equation*}
where
$$
\theta:= \frac{q_*-1}{p-2}+\frac{q}{p}-1=\frac{(N+1)(p-1)((N+1)p-2N)}{p(p-2)}>0
$$
for $p>2$, taking into account that $q_* = p - N/(N+1)$. Consequently, there exists
$$
A_{sub}=\left[(N+1)pB_0^{q/p}\eta^{-(q-1)/(p-1)}\right]^{1/\theta}>0,
$$
such that $-s\mathcal{L}B_A\leq 0$ in $(1,\infty)\times\real^N$ for
$A\in(0,A_{sub})$, ending the proof.
\end{proof}

A first consequence of Lemma~\ref{lem.below} is the optimality of the temporal decay estimates established in Lemma~\ref{lem.above}.

\begin{proposition}\label{pr.lowerbound}
There are $A_0\in (0,A_{sub})$ and $t_0>0$ such that
\begin{equation}
\| u(t)\|_\infty \ge A_0 (e+t)^{-N\eta} \log(e+t)^{-p\eta(N+1)}\ , \qquad t>0\ , \label{phl0}
\end{equation}
and
\begin{equation}
u(t_0,x) \ge B_{A_0}(x)\ , \qquad x\in\real^N\ . \label{phl00}
\end{equation}
\end{proposition}

\begin{proof}
Since $u_0\not\equiv 0$ and is continuous by \eqref{initdata1},
there are $x_0\in\real^N$, $\varrho_0>0$, and $\varepsilon_0>0$ such
that
\begin{equation}
u_0(x) \ge \varepsilon_0\ , \qquad x\in B(x_0,\varrho_0)\ . \label{phl1}
\end{equation}
Introducing $\tilde{u}_0(x) := u_0(x+x_0)$, $x\in\real^N$, and
denoting the corresponding solutions to \eqref{eq1} and
\eqref{eq.resc} by $\tilde{u}$ and $\tilde{w}$, respectively, we
observe that the invariance of \eqref{eq1} with respect to
translations entails that $\tilde{u}(t,x)=u(t,x+x_0)$, while we
infer from \eqref{phl1} that
\begin{equation}
\tilde{u}_0(x) \ge B_A(x)\ , \qquad x\in \real^N\ , \label{phl2}
\end{equation}
provided
\begin{equation}
A \le \varepsilon_0^{(p-2)/(p-1)} \;\;\text{ and }\;\; A \le B_0 \varrho_0^{p/(p-1)}\ . \label{phl3}
\end{equation}
Choosing $A\in (0,A_{sub})$ satisfying \eqref{phl3}, we deduce from Lemma~\ref{lem.below}, \eqref{phl2}, and the comparison principle that
$$
\tilde{w}(s,y) \ge B_A(y)\ , \quad (s,y)\in (1,\infty)\times\real^N\ .
$$
Coming back to $u$, we realize that
\begin{equation}
(e+t)^{N\eta} \log(e+t)^{p(N+1)\eta}\ u(t,x) \ge B_A\left( \frac{(x-x_0) \log(e+t)^{(p-2)(N+1)\eta}}{(e+t)^\eta} \right)\label{phl4}
\end{equation}
for $(t,x)\in (0,\infty)\times\real^N$. A first consequence of \eqref{phl4} with $x=x_0$ is the lower bound \eqref{phl0}. We also infer from \eqref{phl4} with $x=0$ that
$$
(e+t)^{N\eta} \log(e+t)^{p(N+1)\eta}\ u(t,0) \ge B_A\left( \frac{(x_0) \log(e+t)^{(p-2)(N+1)\eta}}{(e+t)^\eta} \right)
$$
and the right-hand side of the above inequality is positive provided
$t$ is large enough. Therefore there is $t_0>0$ such that
$u(t_0,0)>0$, and we argue as in the proof of \eqref{phl2} to
complete the proof of \eqref{phl00}, possibly taking a lower value
of $A$ if necessary.
\end{proof}

\section{Convergence. Proof of Theorem~\ref{th.1}}\label{sec.final}

Thanks to the outcome of Sections~\ref{sec.sd} and~\ref{sec.scaling} we are in a position to prove Theorem~\ref{th.1}. To this end, we follow the lines of the analysis developed by Galaktionov \& V\'azquez in \cite{GVaz, GVazBook}, the central tool being a stability theorem which is recalled in Section~\ref{sec.stabth} for the reader's convenience.

We fix an initial condition $u_0$ satisfying \eqref{initdata1} and denote the corresponding solution to \eqref{eq1}-\eqref{eq1in} by $u$. By Proposition~\ref{pr.lowerbound}, there is $t_0>0$ and $A_0\in (0,A_{sub})$ such that
\begin{equation}
u(t_0,x) \ge B_{A_0}(x)\ , \qquad x\in\real^N\ . \label{z1}
\end{equation}
We then define $w$ by \eqref{resc.1}-\eqref{resc.2} with $u(\cdot+t_0)$ instead of $u$, that is,
\begin{equation*}
u(t+t_0,x)=(e+t)^{-N\eta}\log(e+t)^{-p\eta(N+1)} w(s,y)\ ,
\end{equation*}
the variables $(s,y)$ being still given by \eqref{resc.1}. We infer from \eqref{z1}, Lemma~\ref{lem.above}, Lemma~\ref{lem.below}, and the comparison principle that, for all $s\ge 1$,
\begin{equation}
\|w(s)\|_{1} + \|w(s)\|_{\infty} + \|\nabla w\|_{\infty} \leq C_6\ , \label{z3}
\end{equation}
and
\begin{equation}
w(s,y) \ge B_{A_0}(y)\ , \quad y\in \real^N\ , \;\;\text{ and } \;\; w(s,y) = 0\ , \quad |y|\ge R_0\ . \label{z4}
\end{equation}

We define the set
$$
X := \left\{ z\in L^1(\real^N)\cap BC(\real^N)\ :\ z(y) \ge B_{A_0}(y)\ , \ y\in \real^N \;\;\text{ and }\;\; \|z\|_1 \le C_6 \right\}\ ,
$$
which is a complete metric space for the distance induced by the $L^\infty$-norm, the parameters $C_6$ and $A_0$ being given in \eqref{z3} and \eqref{z4}, respectively. We also set
$$
\mathcal{S} := \{ w \}\ ,
$$
and deduce from \eqref{resc.1}, \eqref{z3}, \eqref{z4}, and the properties of $u$ that $w\in C([0,\infty); X)$. We now check that the set $\mathcal{S}$ enjoys the three properties \textbf{(H1)}-\textbf{(H3)} required to apply the stability result from \cite{GVaz, GVazBook} recalled in Theorem~\ref{th.stab} below. In our setting, the non-autonomous operator is the operator $\mathcal{L}$ defined in \eqref{eq.rescL} and its autonomous counterpart $L$ is defined in \eqref{eq.autonL}, the associated evolution equations being \eqref{eq.resc} and \eqref{eq.auton}, respectively.

Clearly, $w$ is a solution to \eqref{eq.resc} and it readily follows from the $W^{1,\infty}$-bound \eqref{z3}, the uniform localization of the support \eqref{z4}, and the Arzel\'a-Ascoli theorem that $\{ w(s)\}_{s\ge 0}$ is compact in $L^\infty(\real^N)$, so that \textbf{(H1)} is satisfied.

We next infer from the same properties \eqref{z3} and \eqref{z4} that, for $(s,y)\in(0,\infty)\times\real^N$,
\begin{align*}
\left|\mathcal{L}w(s,y)-Lw(s,y)\right| & \le \frac{C}{s} \left[ \|\nabla w(s)\|_\infty^{q_*} + \|w(s)\|_\infty + |y| \|\nabla w(s)\|_\infty \right] \\
& \le \frac{C}{s} \left[ C_6^{q_*} + C_6 + R_0 C_6 \right]\ .
\end{align*}
Therefore,
\begin{equation*}
\left\| \mathcal{L}w(s)-Lw(s) \right\|_\infty \leq\frac{C_7}{s}\ , \quad s>0\ , 
\end{equation*}
from which \textbf{(H2)} follows.

Finally, we fix $A_1\in (A_0,\infty)$ such that $A_1>B_0 R_0^{p/(p-1)}$ and $\| B_{A_1}\|_1 \ge C_6$ and set
\begin{equation*}
\Omega := \left\{ B_A\ : \ A\in [A_0,A_1] \right\}\ . 
\end{equation*}
Clearly $\Omega$ is a non-empty and compact subset of $X$. Since \textbf{(H3)} concerns only the autonomous operator \eqref{eq.auton}, which also arises from the standard $p$-Laplacian equation via a self-similar change of variable, the uniform stability of $\Omega$ holds true as a by-product of classical results for the $p$-Laplacian equation (see \cite[Section~6]{GVaz} and also \cite[Section
4.6]{GVazBook}).

We have thus checked the validity of \textbf{(H1)}-\textbf{(H3)} and may thus apply the stability theorem (Theorem~\ref{th.stab}) to conclude that the $\omega$-limit set (for the topology of the uniform convergence) of $w$ is included in $\Omega$, that is,
\begin{equation}
\omega(w) := \left\{
\begin{array}{l}
\bar{w} \in X\ :\ \text{ there is a sequence } (s_j)_{j\ge 1}\ , s_j\to\infty\ ,\\
\text{ such that } \lim_{j\to \infty} \|w(s_j)-\bar{w} \|_\infty = 0
\end{array}
 \right\} \subset \Omega\ . \label{z7}
\end{equation}

\medskip

\textbf{Mass analysis. Uniqueness of the limit.} It remains to show
that the asymptotic profile is in fact a uniquely determined Barenblatt profile
with the parameter $A_*$ as in Theorem~\ref{th.1}. To this end, we
perform a mass analysis along the lines of a similar argument in \cite{GVaz}. We first observe that classical properties of the Barenblatt profiles ensure that, given $s\ge 0$, there is a unique $A(s)>0$ such that
\begin{equation}
\Theta(s):=\|w(s)\|_{1} = \|B_{A(s)}\|_1\ . \label{7}
\end{equation}
More precisely,
\begin{equation}\label{interm8}
\Theta(s) = \frac{N}{N+1} \omega_N I_1 A(s)^{(p-1)/[p(p-2)\eta]}\ ,
\end{equation}
where
\begin{equation}\label{interm9}
I_1 := (N+1)
\int_0^{\infty}\left(1-B_0r^{p/(p-1)}\right)_+^{(p-1)/(p-2)}r^{N-1}\,dr
\end{equation}
and $\omega_N$ denotes the volume of the unit ball of $\real^N$, see \cite[Section 11.4.1]{VazquezSmoothing} for instance. Since
$$
\|B_{A_1}\|_1 \ge C_6 \ge \Theta(s) \ge \|B_{A_0}\|_1\ , \quad s\ge 0\ ,
$$
by \eqref{z3}, \eqref{z4}, and the choice of $A_1$, we deduce from the monotonicity of $A\mapsto \|B_A\|_1$ and \eqref{7} that
\begin{equation}
A_0 \le A(s) \le A_1\ , \quad s\ge 0\ . \label{z8}
\end{equation}
In addition, integrating \eqref{eq.resc} with respect to space shows that, for $s\ge 0$,
\begin{equation}\label{interm7}
\frac{d\Theta}{ds}(s)=\frac{G(w(s))}{s} \;\;\text{ with }\;\; G(z) := (N+1) \|z\|_1 - \|\nabla z\|_{q_*}^{q_*}\ .
\end{equation}
Owing to \eqref{interm7}, \eqref{interm8}, and the regularity properties of $w$, we realize that $A\in C([0,\infty))\cap C^1((0,\infty))$.

We next claim that
\begin{equation}
\lim_{s\to\infty} \| w(s) - B_{A(s)}\|_\infty = 0\ . \label{z9}
\end{equation}
Indeed, assume for contradiction that there are an increasing sequence $(s_j)_{j\ge 1}$ of positive real numbers, $s_j\to\infty$, and $\varepsilon>0$ such that
\begin{equation}
\| w(s_j) - B_{A(s_j)}\|_\infty \ge \varepsilon\ , \quad j\ge 1\ . \label{z10}
\end{equation}
On the one hand, we infer from \eqref{z3}, \eqref{z4}, \eqref{z7}, and the Arzel\'a-Ascoli theorem that there are a subsequence of $(s_j)_{j\ge 1}$ (not relabeled) and $\bar{A} \in [A_0,A_1]$ such that
$$
\lim_{j\to\infty} \|w(s_j)-B_{\bar{A}}\|_\infty = 0\ .
$$
On the other hand, it follows from \eqref{z8} that, after possibly extracting a further susbequence, we may assume that there is $A_\infty\in [A_0,A_1]$ such that $A(s_j)\to A_\infty$ as $j\to \infty$. This readily implies that
$$
\lim_{j\to\infty} \|B_{A(s_j)}-B_{A_\infty}\|_\infty = 0\ .
$$
We may then let $j\to\infty$ in \eqref{z10} to conclude that
\begin{equation}
\|B_{\bar{A}} - B_{A_\infty}\|_\infty \ge \varepsilon\ . \label{z11}
\end{equation}
Now,
$$
\| B_{\bar{A}} \|_1 = \lim_{j\to \infty} \Theta(s_j) = \lim_{j\to\infty} \left\| B_{A(s_j)} \right\|_1 = \| B_{A_\infty} \|_1\ ,
$$
so that $\bar{A}= A_\infty$, which contradicts \eqref{z11}. We have
thus proved \eqref{z9}.

We next infer from \eqref{z3} and \eqref{z4} that, introducing
\begin{align*}
f(s,y) & := \eta y\cdot \nabla w(s,y) + \eta N w(s,y) \\
& \quad - \frac{1}{s} \left[ |\nabla w(s,y)|^{q_*} - p \eta (N+1) w(s,y) + (p-2) \eta (N+1) y\cdot \nabla w(s,y) \right]
\end{align*}
for $(s,y)\in (0,\infty)\times \real^N$, Eq.~\eqref{eq.resc} reads
$$
\partial_s w - \Delta_p w = f \;\;\text{ in }\;\; (1,\infty)\times B(0,R_0+1)
$$
with $w\in L^\infty(1,\infty; W^{1,\infty}(B(0,R_0+1)))$ and $f\in L^\infty((1,\infty)\times B(0,R_0+1))$. We then infer from \cite[Theorem~1.1]{DBF85} that there are $C_8>0$ and $\nu\in (0,1)$ such that
$$
|\nabla w(s_1,y_1) - \nabla w(s_2,y_2)| \le C_8 \left( |y_1-y_2|^\nu + |s_1-s_2|^{\nu/2} \right)
$$
for all $s_2\ge s_1 \ge 2$ and $(y_1,y_2)\in B(0,R_0)\times B(0,R_0)$. Combining the above property with \eqref{z3} we deduce that $\{\nabla w(s)\}_{s\ge 1}$ is bounded and equicontinuous in $C(B(0,R_0);\real^N)$ and thus compact in that space by the Arzel\'a-Ascoli theorem. Recalling \eqref{z4} and \eqref{z9} we conclude that
\begin{equation}
\lim_{s\to\infty} \left( \|w(s)-B_{A(s)}\|_r + \| \nabla w(s) - \nabla B_{A(s)}\|_r\right) = 0\ , \quad r\in [1,\infty]\ . \label{z12}
\end{equation}
An immediate consequence of \eqref{z12} is that
\begin{equation}
\lim_{s\to\infty} \left| G(w(s)) - G(B_{A(s)}) \right| = 0\ , \label{z13}
\end{equation}
while the explicit formula for $B_{A(s)}$ gives
\begin{equation}
G(B_{A(s)}) = g(A(s)) \label{z14}
\end{equation}
with
$$
g(a) := N\omega_N I_1 a^{(p-1)/\eta p(p-2)} - N \omega_N I_2 a^{(N+2)(p-1)/(N+1)\eta p(p-2)}\ , \quad a>0\ ,
$$
where $I_1$ is defined in \eqref{interm9}, and
$$
I_2:=\eta^{q/(p-1)}\int_0^{\infty}\left(1-B_0r^{p/(p-1)}\right)_+^{q/(p-2)}r^{((p-1)(N-1)+q)/(p-1)}\,dr\ .
$$
We observe that $g$ vanishes only once in $(0,\infty)$, for $a=A_*$,
which is defined in \eqref{interm0} and reads
$$
A_*:=\left(\frac{I_1}{I_2}\right)^{p(p-2)(N+1)\eta/(p-1)}\ ,
$$
with the notation introduced in this section. In fact,
\begin{equation}
g(a)<0 \;\;\text{ for }\;\; a>A_* \;\;\text{ and }\;\; g(a)>0 \;\;\text{ for }\;\; a\in (0,A_*)\ . \label{z15}
\end{equation}

Thanks to \eqref{z7}, \eqref{interm7}, \eqref{z12}, \eqref{z13},\eqref{z14}, and
\eqref{z15}, we are in a position to proceed as in the proof of
\cite[Proposition~5.2]{GVaz} to establish that $\omega(w)=\{ B_{A_*}
\}$. Undoing the rescaling \eqref{resc.1}-\eqref{resc.2}, this
property readily gives Theorem~\ref{th.1}.

We finally note that Proposition~\ref{pr.posexp} follows at once from \eqref{z4} after translating these properties in terms of $u$.

\appendix
\section{The stability theorem}\label{sec.stabth}

We briefly recall here for the reader's convenience the stability
theorem introduced by Galaktionov and V\'azquez in \cite{GVaz,
GVazBook} and used in Section~\ref{sec.final}. As a general
framework, consider a non-autonomous evolution equation
\begin{equation}\label{eq.pert}
\partial_s \vartheta = \mathcal{L}\vartheta\ ,
\end{equation}
that can be seen as a \emph{small perturbation} of an autonomous
evolution equation with good asymptotic properties
\begin{equation}\label{eq.unpert}
\partial_s\Phi = L\Phi\ ,
\end{equation}
in the sense described below. We consider a set $\mathcal{S}$ of solutions $\vartheta\in C([0,\infty);X)$ to \eqref{eq.pert} with values in a complete metric space $(X,d)$. We assume that:

\begin{itemize}
\item[\textbf{(H1)}] For each $\vartheta\in\mathcal{S}$, the orbit $\{\vartheta(t)\}_{t>0}$ is relatively compact in $X$. Moreover, if we let
$$
\vartheta^{\tau}(t):=\vartheta(t+\tau)\ , \quad t\ge 0\ , \quad
\tau>0,
$$
then $\{\vartheta^{\tau}\}_{\tau>0}$ is relatively compact in
$C([0,T];X)$ for any $T>0$.

\item[\textbf{(H2)}] Let $\vartheta\in\mathcal{S}$ for which there is a sequence of positive times $(t_k)_{k\ge 1}$, $t_{k}\to\infty$, such that $\vartheta(\cdot+t_k) \longrightarrow \Theta$ in $C([0,T];X)$ as $k\to\infty$ for any $T>0$. Then $\Theta$ is a solution to \eqref{eq.unpert}.

\item[\textbf{(H3)}] Define the $\omega$-limit set $\Omega$ of \eqref{eq.unpert} in $X$ as the set of $f\in X$ enjoying the following property: there are a solution $\Phi\in C([0,\infty);X)$ to \eqref{eq.unpert} and a sequence of positive times $(t_k)_{k\ge 1}$ such that $t_k\to\infty$ and $\Phi(t_k)\longrightarrow f$ in $X$. Then $\Omega$ is non-empty, compact and uniformly stable, that is: for any $\e>0$, there exists $\delta>0$ such that if $\Phi$ is a solution to \eqref{eq.unpert} with $d(\Phi(0),\Omega)\leq\delta$, then $d(\Phi(t),\Omega)\leq\e$ for any $t>0$.
\end{itemize}

The stability theorem (also known as the S-theorem) then reads:
\begin{theorem}\label{th.stab}
If \textbf{(H1)-(H3)} above are satisfied, then the $\omega$-limit
set of any solution $\vartheta\in \mathcal{S}$ is contained in $\Omega$.
\end{theorem}

For a detailed proof we refer the reader to \cite{GVaz, GVazBook}.

\bibliographystyle{plain}


\end{document}